\documentclass[11pt,oneside]{amsart}
\usepackage{amsmath}
\usepackage{amssymb}
\usepackage{eucal}
\usepackage{amsthm}
\usepackage{amscd}
\usepackage{hyperref}
\usepackage{soul}
\usepackage{url}
\usepackage{skull}
\usepackage{mathrsfs}

\usepackage{tikz}
\usepackage{tikz-cd}
\usetikzlibrary{shapes.geometric, arrows}
\usetikzlibrary{shapes,arrows}
\usepackage[latin1]{inputenc}
\usetikzlibrary{shapes,arrows}
\usetikzlibrary{shapes.multipart}
\usepackage{verbatim}

\usepackage{amssymb}
\usepackage[]{amsmath, amsthm, amsfonts,graphicx, amscd}
\usepackage[all,cmtip]{xy}
\input amssym.def \input amssym
\usepackage{forest}
\usepackage{mdframed}
\usepackage{framed}

\newcommand{\acl}{\operatorname{acl}}

\newtheorem{thm}{Theorem}[section]

\newtheorem{mainthm}{Theorem}

\newtheorem{prop}[thm]{Proposition}

\newtheorem {lem}[thm]{Lemma}

\theoremstyle{remark}
\newtheorem{rem}[thm]{Remark}

\newtheorem{np*}{Non-Proof}

\theoremstyle{definition}
\newtheorem{defn}[thm]{Definition}

\theoremstyle{plain}

\newcommand{\td}{\text{tr.deg.}}

\def\Ind{\setbox0=\hbox{$x$}\kern\wd0\hbox to 0pt{\hss$\mid$\hss} \lower.9\ht0\hbox to 0pt{\hss$\smile$\hss}\kern\wd0}
\def\Notind{\setbox0=\hbox{$x$}\kern\wd0\hbox to 0pt{\mathchardef \nn=12854\hss$\nn$\kern1.4\wd0\hss}\hbox to 0pt{\hss$\mid$\hss}\lower.9\ht0 \hbox to 0pt{\hss$\smile$\hss}\kern\wd0}
\def\ind{\mathop{\mathpalette\Ind{}}}
\def\nind{\mathop{\mathpalette\Notind{}}}
\numberwithin{equation}{section}

\newcommand{\m}{\mathbb }
\newcommand{\mc}{\mathcal }

\newcommand{\tp}{\operatorname{tp}}

\begin{document}

\title{When any $4$ solutions are independent}
\author{James Freitag}
\date{\today} 

\begin{abstract}
We show that if any four distinct solutions of a rational difference equation are algebraically independent, then any number of distinct solutions to the equation are independent. A nontrivial variant of this result is given for autonomous difference equations or algebraic dynamical systems, where we show the degree of nonminimality is at most two. The results have a natural interpretation in terms of invariant or periodic subvarieties of algebraic dynamical systems and $\sigma$-varieties. Surprisingly, the proofs of these results rely on the classification of finite simple groups. 
\end{abstract}

\maketitle

\section{Introduction}
Let $k$ be an algebraically closed field of characteristic zero. An \emph{algebraic dynamical system over $k$} is a pair $(X, \phi )$, where $X$ is a variety over $k$ and $\phi :X \rightarrow X$ is a dominant rational map over $k$. More generally, when $\sigma$ is an endomorphism of $k$, a \emph{$\sigma$-variety over $(k,\sigma)$} is a pair $(X,\phi)$ where $X$ is a $k$-variety and $\phi:X \rightarrow X^\sigma$ is a dominant rational map over $k$. Sometimes we refer to algebraic dynamical systems as the \emph{autonomous case} of rational $\sigma$-varieties, following \cite{ChHr-AD1,chatzidakis2008difference,kamensky2024binding}. Most sources consider the special case in which $\phi$ in the above definitions is assumed to be a morphism (see e.g. \cite{pink2004psi} for the nonautonomous case). A not necessarily irreducible subvariety $Y \subset X$ is called \emph{weakly invariant} if $\phi (Y) \subset Y^\sigma$ and $Y$ is not contained in the indeterminacy locus of $\phi$. \footnote{Note that many sources use some variation of this terminology. For instance, \cite{AliceTom} calls this notion weak skew invariance.} $Y$ is \emph{invariant} if $\phi (Y) = Y^\sigma .$

Our main results relate to the invariant subvarieties of products of the dynamical system with itself. We denote the $n$-fold product of $X$ equipped with the map $\phi$ on each copy of $X$ by $(X^n,\phi^{\times n})$. Morphisms in this category from $(X , \phi)$ to $(Y,\psi)$ are given by a morphism $f: X \rightarrow Y$ of algebraic varieties such that the following diagram commutes:
\[
\begin{tikzcd}
X \arrow[r, "\phi"] \arrow[d, "f"'] & X^\sigma \arrow[d, "f^\sigma"] \\
 Y \arrow[r, "\psi"'] & Y^\sigma
\end{tikzcd}
\]
where we denote by $f^\sigma$, the function which has $\sigma$ applied to the coefficients of $f$.

Given a $\sigma$-variety, $(X,\phi)$, the $\sigma$-variety $(X^n, \phi^{\times n} )$ always has certain proper irreducible invariant subvarieties, and we next describe these \emph{obvious sources}. Products of invariant subvarieties of $X$ are invariant in $X^n$. The subvariety obtained by simply setting the coordinates of two of the copies of $X$ equal to each other is also clearly invariant. We call such a subvariety a \emph{generalized diagonal}. 

More generally, when $X$ is defined over the fixed field of $\sigma $ and $g = f^{\sigma^n}$, then the graph of $f^{\sigma^{n-1} } \circ f^ {\sigma^{n-2} } \circ \ldots \circ f$ is invariant for $(X,f) \times (X,g)$. Medvedev and Scanlon \cite{AliceTom} call this skew-iteration. Since we are considering the general case of subvarieties of products of $(X,f)$, when our system is nonautonomous, this won't be a source of invariant subvarieties (besides the generalized diagonal) except under very specific circumstances when the action of $\sigma$ on the coefficients of $f$ has a finite orbit. When $f$ has coefficients in the fixed field itself, however, the graphs of iterates of $f$ are invariant on $(X,f) \times (X,f)$. For certain special choices of $f$, one might have additional less obvious invariant subvarieties. For instance, $f$ might commute with some other map or may be written as a composition in some nontrivial way. 

We are interested in the existence of irreducible invariant subvarieties of $X^n$ which do not arise from these obvious sources (or their proper subvarieties). So, we call invariant subvariety of $Y \subset X^n$ is called \emph{non-diagonal} if it is 

\begin{itemize}
    \item not contained in a generalized diagonal 
    \item for each $i$, the projection map to the $i^{th}$ copy of $V$, $\pi_i : Y \rightarrow V$ is dominant
\end{itemize} Our main result governs the appearance of non-diagonal invariant subvarieties in the nonautonomous case.

\begin{mainthm} \label{mainthm} Fix $(X,f)$ a $\sigma$-variety. Suppose that for some $n$, $X^n$ has a \emph{non-diagonal} proper irreducible invariant subvariety. Then there is already a non-diagonal proper irreducible invariant subvariety of $X^4$. 
\end{mainthm}

As we've mentioned, in the autonomous case, the additional source of invariant subvarieties given by iterates of $f$ means that the previous result is trivial since already the graph of an iterate of $f$ is non-diagonal on $X^2$. We will prove a result about the appearance of an interesting invariant subvarieties in the autonomous case, but first we need some notation. 

Note that all of the \emph{obvious} invariant subvarieties from our above sources share a common feature - let $\pi_{\underline{i}}: X^n \rightarrow X^{n-1}$ denote projection away from the $i^{th}$ copy of $X$. If $Y$ is an invariant subvariety which projects dominantly onto each copy of $X$ and is formed by intersecting generalized diagonals and graphs of iterates of $f$ as above, then the fiber above a generic point in each projection $\pi_{\underline{i}}$ will either be zero-dimensional or of $\dim X$. 

We call a non-diagonal irreducible invariant subvariety $Y \subset X^n$ a \emph{nonminimal forking} subvariety if for some $i = 1, \ldots , n$, the relative dimension of $\pi_{\underline{n}}$ on $Y$ is positive, but less than $\dim X$. That is, $0<dim (\pi_{\underline{n}}^{-1}(Z)  / Z) < \dim X ,$
where $Z=\pi_{\underline{n}}(Y).$ Of course, this notion only makes sense for $X$ of dimension at least $2$. By their dimension properties - nonminimal forking subvarieties can not be built from any of the obvious sources of invariant subvarieties we described above. A nonminimal forking subvariety of $X^n$ corresponds to a forking extension of a generic solution of the difference equation given by $(X,\phi)$, a perspective we will later explain and exploit. 

\begin{mainthm} \label{nmdeg}
Suppose $(X, \phi)$ is an autonomous algebraic dynamical system. If $(X^n, \phi^{\times n})$ has a nonminimal forking subvariety for some $n$, then already $(X^3, \phi ^{\times 3})$ has a nonminiml forking subvariety. 
\end{mainthm}

In the non-autonomous case, we don't know any absolute bound at the moment\footnote{During the preparation of this manuscript in various talks, I conjectured that $\dim X+4$ in the previous theorem can be improved to $4$. In a manuscript in preparation with Jaoui, Karhum\"aki, and Ramsey, we prove this conjecture.}: 

\begin{mainthm} \label{nmdegnon} 
Suppose $(X, \phi)$ is a $\sigma$-variety. If $(X^n, \phi^{\times n})$ has a nonminimal forking subvariety, then already $(X^{\dim X+4}, \phi ^{\times (\dim X+4)})$ has a nonminimal forking subvariety. 
\end{mainthm}

The proofs of the last two results utilize the work of \cite{kamensky2024binding}, where, as in \cite{freitag2021bounding}, the degree of generic transitivity of an automorphism group was related to the problem of weak nonorthogonality of types. Unfortunately, presuming Question 5.7 of \cite{kamensky2024binding} has a positive answer, improving the above bound would require some additional ingredient or approach. Question 5.7 of \cite{kamensky2024binding} does have a positive answer in the case of differential equations - see \cite{freitag2021bounding}. 

Our approach to these results uses the model theory of the model companion of characteristic zero difference fields, which we denote by $ACFA$. Associated with any $\sigma$-variety are the solutions of a difference equation, and the type of a generic solution of this equation is what we call a \emph{rational type}, following \cite{kamensky2024binding}. By $k \langle a \rangle$, we denote the difference field over $k$ generated by $a$, $k (a , \sigma (a), \sigma^2 (a) , \ldots )$. Relative to the theory $ACFA$, forking independence is characterized in terms of transcendence degree, $a \ind _k b$ if and only if $\td _k (k\langle a \rangle ) = \td _{k \langle b \rangle } (k\langle a \rangle )$. 

For rational types, forking dependence has a natural interpretation. Let $a \models p$ and $b \models q$, the rational types of $(V, \phi )$ and $(W, \psi)$. Suppose $a \nind _k b$. Let $U \subset V \times W$ be the zero set of the difference idea $I$ which includes any $\sigma$-polynomial which vanishes at $(a,b)$. Then $U$ is a non-diagonal irreducible invariant subvariety of $V \times W$ over $k$. Conversely, any non-diagonal irreducible invariant subvariety over $k$ induces an algebraic relation between the elements $a,b$ over $k$. Then the quantifier free type of $a$ over $k\langle b \rangle=k(b)$ is a forking extension of $p = \tp_{qf} (a/k)$. This connection is how our main results are stated and proved in section \ref{any3sec}. 
Similarly, nonminimal forking subvarieties of $X^n$ correspond to \emph{non-algebraic} forking extensions of the type $p$. So, the existence of invariant forking subvarieties has a natural interpretation in terms of the \emph{degree of nonminimality} of the type p, see \cite{baldwin2024simple, freitag2023degre}. Theorem \ref{nmdeg} says that the degree of nonminimality of a type over the fixed field is at most $2$. Theorem \ref{nmdegnon} says that  the degree of nonminimality of a general finite rank type is at most $\dim p +3$. 

\begin{rem}
In the particular formulation of the above results, the issue of the invariant subvarieties being irreducible is slightly delicate. We've stated the results in geometric language above, but our proofs use the model theory of ACFA. Let $\mc U \models ACFA$ and consider the solutions of the difference equation: $$(V, \phi)^\sharp = \{ a \in V(\mc U) \m | \, \sigma (a) = \phi (a) \}.$$
An irreducible variety $W \subset V$ is invariant if and only if $Y \cap (V, \phi)^\sharp$ is Zariski-dense in $Y$. This correspondence is at the heart of our results, but it fails when $Y$ is not irreducible. We emphasize that our results are, on a fundamental level, about the rational type which is a generic solution of the difference equation $(X, \phi)^\sharp$ of a $\sigma$-variety or dynamical system. The connection to the system itself can be somewhat delicate.  Many natural questions about the dynamical system itself won't manifest in terms of these generic solutions, and thus won't be amenable to attack by our methods without significant modification. 

Consider for instance the weakly invariant subvariety of $(X^2 , f^{\times 2})$ given by $f(x)=f(y)$. The subvariety is invariant, but the non-diagonal components (the components of $\frac{f(x)-f(y)}{x-y}=0$) are not. These components won't have generic points in $\mc U$ solving the system of difference equations. Examples in which the image of $f$ is dense in each component of $Y^\sigma$ can be obtained from periodic subvarieties, for which versions of our results will apply as we explain next.
\end{rem}

\subsection{From invariant to periodic} When $\sigma$ is not the identity on $k$, $f$ is not a self-map of $X$. Nevertheless, there is a sensible notion of periodicity of subvarieties. Let $(V,f)$ be a $\sigma$-variety. 
A subvariety $X$ of $V$ is called \emph{periodic} of order $d$ if $f^{\diamond d} (X) = X^{\sigma^d}$ where $f^{\diamond d}= f^{\sigma ^{d-1}} \circ f^{\sigma ^{d-2}} \ldots f^\sigma \circ f : V \rightarrow V^{\sigma^n}$. In the autonomous situation, when $\sigma$ is the identity, this matches the standard notion of periodic. 

Then $(V,f^{\diamond d})$ is a $\sigma ^d$-variety. Working in (the reduct) $\mc U$ in the language of rings together with a symbol $\tau$ for $\sigma^d$, the structure is a model of $ACFA$ in this language by \cite[Corollary of page 3013]{Chatzidakis99modeltheory}. Additionally, $X$ is an invariant subvariety of the $\tau$-variety $(V, f^{\diamond d})$. As such, each of our main results applies to $(V,f^{\diamond d})$ working in the structure $\mc U$ with $\tau $ and the language of rings. Thus, one can replace invariant by periodic of some fixed period in each of the statements of the main results above. We will refrain from returning to this point in the subsequent sections of the paper. 

\subsection{Outline of the paper} 
In section \ref{two}, we review some recent related results regarding invariant subvarieties of algebraic dynamical systems. In section \ref{three} we cover the facts from $ACFA$ which we need. In section \ref{four} we prove the results we need coming from pseudofinite permutation group theory. In section \ref{any3sec} we prove our main results. 

\subsection{Differential algebraic geometry}
In the setting of differential algebraic geometry, the analogous result \cite{freitag2022any} to our main theorem and its subsequent strengthening \cite{freitag2023degre} has seen applications via the via its contrapositive form. The strong minimality of several classes of equations has been established via showing that any three solutions are independent or by using the fact that that the degree of nonminimality is small. For instance, Nagloo and the author \cite{freitag2025algebraic} give a new proof of the irreducibility of certain Painlev\'e equations. Nagloo and Duan \cite{duan2025algebraic} apply the criterion to Lotka-Volterra systems. Devilbiss and the author \cite{devilbiss2023generic} apply the criterion to solve the nonconsant case of Poizat's conjecture on generic differential equations. Casale, Nagloo, Devilbiss and the author \cite{casale2022strong} apply the criterion to give a new proof of the strong minimality of the equations of certain automorphic functions (e.g. the $j$-function). We hope the results of this paper open up a similar possibility for difference equations and algebraic dynamics.

\subsection{Acknowledgements} Thanks to the organizers of the American Institute of Mathematics workshop Motives and Mapping Class Groups, where this work was initially begun. Thanks in particular to Daniel Litt and Aaron Landesman. Thanks also to Nick Ramsey, Carlos Arreche, Ronnie Nagloo, Holly Krieger, and especially Remi Jaoui for conversations during the preparation of this manuscript.

\section{Dynamics and invariant subvarieties} \label{two}
The results of several recent works characterize the invariant subvarieties of special classes of rational $\sigma$-varieties. Characterizing invariant subvarieties of products of curves has, in particular, played an important role in numerous recent results in dynamics. 

For instance, when $X=\m A^n$ and $\phi$ is given by single variable polynomials, $(x_1, \ldots , x_n) \mapsto (f_1(x_1), \ldots , f_n(x_n))$ of degree at least $2$, Medvedev and Scanlon \cite{AliceTom} show that invariant subvarieties can be built from invariant plane curves of $(\m A^2, (f_i,f_j))$ for some $i,j$. Medvedev and Scanlon then make a detailed study of such invariant plane curves. There are partial generalizations to higher dimensional dynamics \cite{xie2024algebraicity}. Bell, Moosa and Satriano \cite{bell2024invariant} study when an algebraic dynamical system $(X,\phi)$ admits a nonconstant \emph{invariant rational map}, that is a map to the dynamical system $(\m A^1 , id)$: 
\[
\begin{tikzcd}
X \arrow[r, "\phi"] \arrow[d, "f"'] & X^\sigma \arrow[d, "f^\sigma"] \\
\m A^1 \arrow[r, "id"'] & \m A^1
\end{tikzcd}
\]

Such an invariant rational function yields a family of invariant hypersurfaces of $X$ (by considering the fibers of the function). In fact, \cite{bell2010dixmier} shows that this is the only way in which families of invariant hypersurfaces arise. It can happen that $(X,\phi)$ has no invariant rational function over $k$, but there is such a function over a $\sigma$-field $K$ extending $k_1.$ This is a central problem of study in \cite{bell2024invariant} in the autonomous case and \cite{kamensky2024binding} in the general case. When there is such an invariant rational function over a $\sigma$-field extension, there is a commuting diagram of the form:
\[
\begin{tikzcd}
  (X \times Z , \phi \times \psi ) \arrow[rr, "g"' ] \arrow[dr, ] & & (\m A^1 \times Z , id \times \psi ) \arrow[dl, ]\\
  & ( Z , \psi ) &
\end{tikzcd}
\]

Whenever there is such a diagram for $(X,\phi)$, there is a diagram in which the role of $(Z, \psi)$ can be played by $(X^n , \phi^{\times n})$. The main result of \cite{bell2024invariant} shows that in the autonomous case assuming some power of $(X,\phi)$ has an invariant rational function, then already $(X^2 , \phi^{\times 2})$ does. Kamensky and Moosa \cite{kamensky2024binding} prove a weaker result in the more general non-autonomous setting, bounding the power of $(X, \phi)$ by $\dim X +3$. The fibers of an invariant rational function give rise to a family of invariant hypersurfaces - which are nonminimal forking subvarieties in the notation of our introduction above. 

In \cite{xie2024algebraicity}, Xie studies invariant subvarieties for products of endomorphisms of projective varieties $f_i : X_i \rightarrow X_i$ under the strong assumption that every invariant subvariety of the product is the product of invariant subvarieties of the factors (or unions of such). In model theoretic terms, the systems $(X_i , f_i)$ are \emph{weakly orthogonal} in his setting.\footnote{In \cite{xie2024algebraicity} and \cite{AliceTom}, the terminology \emph{almost orthogonal} is used, but we prefer the term \emph{weakly orthogonal} as it discourages confusion with the term \emph{almost internality}.} Ghioca and Xie \cite{ghioca2018algebraic, ghioca2020dynamical} study invariant subvarieties of \emph{skew-linear} maps $\phi: X \times \m A^n \rightarrow X \times \m A^n$. 

Invariant subvarieties of particular maps have also had striking applications to classical diophantine results. For instance Pink and Roessler \cite{pink2002hrushovski} characterize invariant subvarieties of isogenies of abelian varieties to characterize subvarieties with Zariski dense sets of torsion points. Krieger and Reschke \cite{krieger2017coho} generalize the result to endomorphisms of abelian varieties. Other classical diophantine results have dynamical analogues. See \cite{10.1215/00127094-1384773, ghioca2017dynamical, ghioca2018dynamical} for cases of the dynamical Andr\'e-Oort. 

One reason motivating the study of invariant subvarieties is the \emph{Zariski dense orbit conjecture}, whose original form goes back to Zhang \cite{zhang2006distributions}. Zhang conjectured that when $X$ is a projective variety and $\phi:X \rightarrow X$ is a polarized endomorphism, then there is a point $x \in X(k)$ such that the forward orbit of $x$ is Zariski dense in $X$. The closure of a forward orbit is invariant, so one might regard the conjecture as describing a certain kind of sparsity of the invariant subvarieties. Medvedev and Scanlon \cite{AliceTom} strengthen the conjecture, dropping the assumption that $X$ is polarized and $\phi $ is projective, but assuming there does not exist a positive dimensional algebraic variety $Y$ and dominant rational map $g : X \rightarrow Y$ for which $g \circ \phi = g$ generically. Such a map gives rise to an invariant rational function, $f$ as above and a family of invariant hypersurfaces. In \cite[Conjecture 1.4]{xie2025existence}, Xie slightly strengthens the conjecture of Medvedev and Scanlon, adding that a point whose orbit is dense can be found in any open subset of $X$. This formulation is hinted at by Medvedev and Scanlon - see \cite[Remark 7.17]{AliceTom} and considered in various works \cite{ghioca2018algebraic}. A polarized dynamical system has no invariant rational function - see for instance the introduction of \cite{xie2025existence}. This condition is necessary for the Zariski-dense orbit conjecture. 

It is often far from obvious that a dynmical system has an invariant rational function. For instance,  $(\m A^4, (x_2, -x_4,x_1-x_1x_2^2, -x_3+x_1x_2x_4))$ has invariant rational function $f=x_1x_4-x_2x_3$. For additional examples of this kind, see \cite[page 5]{bisi2024some}. The collection of invariant rational functions, like the property of possessing a Zariski-dense orbit, is not stable under taking products - for instance, take $\phi: \m A^1 \rightarrow \m A^1$ given by $x \mapsto x+1.$ 

As we have mentioned, the appearance of an invariant rational function on $X^n$ also signals the appearance of a invariant forking subvariety, but the converse is not necessarily true. Kamensky and Moosa \cite{kamensky2024binding} show that if $X^n$ has an invariant rational function for some $n$, then there is already such a function on $X^{\dim X +3}$. They conjecture that this bound can not in general be improved. The analogous conjecture is true in the setting of differential algebraic geometry \cite{freitag2025finite}. Our main result shows that the situation is different for non-diagonal invariant subvarieties. If they ever appear, they  appear already in $X^4$. 

\section{Reviewing rational $\sigma$-varieties and quantifier free types in $ACFA$} \label{three}
Fix $(k,\sigma)$ where $k$ is an algebraically closed  characteristic zero field and $\sigma$ is an automorphism. Let $\mc U \models ACFA$ be a saturated model containing $k$. Let $\mc C$ denote the fixed field $\{u \in \mc U \, | \, \sigma (u)=u \}$ of $\mc U$. We consider difference equations (more generally definable sets) as synonymous with their points in the model $\mc U$. 

We will generally work in the setting of \cite{kamensky2024binding}, which we will briefly review in this section. The basics of the model theory of $ACFA$ have been covered by numerous references in the literature \cite{chatzidakis2005model, chatzidakis2015model, Chatzidakis99modeltheory, kamensky2024binding, AliceTom}, and so we restrict ourselves to the necessary technical results and the less classical aspects which our proofs require. Many of the notions of \cite{kamensky2024binding} are the analog of classical notions from model theory, but some care is required as we are interested in the \emph{quantifier-free} analogs of various notions, but our theory, $ACFA$, does not admit quantifier elimination. 

\begin{defn}
    A rational $\sigma$-variety over $k$ is a pair $(V,\phi)$ with $\phi: V \rightarrow V^\sigma$ a dominant rational map with $V/k$ an absolutely irreducible variety. 
\end{defn}

Associated with any rational $\sigma$-variety are the solutions of a difference equation: 
$$(V, \phi )^\sharp = \{a \in V(\mc U ) \, | \, \sigma (a) = \phi (a) \} $$

Associated to any rational $\sigma$-variety is the quantifier-free type of a (Zariski) generic over $k$ point $a \in V$ such that $\sigma (a) = \phi (a)$. Such a quantifier-free type will be referred to as a \emph{rational type}. The collection of quantifier-free types over $A$ is denoted $S_{qf}(A)$. We say $p \in S_{qf} (A)$ is stationary if for all $B \supset A$ there is a unique nonforking extension $q \in S_{qf} (B)$ - that is, an extension such that for all $a \models q$, $a \ind _A B.$ 

\begin{lem} \cite[Lemma 2.2]{kamensky2024binding}
Quantifier-free types $p \in S_{qf} (A)$ are stationary whenever $A$ is an algebraically closed difference field. 
\end{lem}

The class of rational types is closed under rational maps - that is, when $a \models p$, a rational type over $k$, and $f(x)$ is a rational map, $f(a)$ satisfies a rational type over $k.$ A rational map $f:p \rightarrow q$ of rational types is a $k$-rational map so that for every (equivalently some) $a \models p$, $f(a) \models q$. Two rational types are birationally equivalent if there are rational maps $p \rightarrow q$ and $q \rightarrow p$. The \emph{dimension} of a rational type is the transcendence degree of the field generated by a realization of the type.

Recall that a type $p=\tp (a/k)$ is \emph{weakly orthogonal} to $\mc C$ if $a \ind _k c$ for any tuple of elements $c \in \mc C$ - we write $p \perp ^w  \mc C$. A type $p=\tp (a/k)$ is \emph{orthogonal} to $\mc C$ if any nonforking extension of $p$ is weakly orthogonal to $\mc C$ - we write we write $p \perp  \mc C$. 

\begin{lem} \label{weakortho} \cite[Proposition 2.6]{kamensky2024binding}
Let $a$ satisfy a rational $\sigma$-type over $k$. Then $\tp(a/k)$ is weakly orthogonal to $\mc C$ if $k(a) \cap \mc C \subset k^{alg}$.  
\end{lem}

By the previous fact, weak orthogonality depends only on the quantifier-free type of $a$ over $k$. Thus, it makes sense to discuss (weak) orthogonality of rational types (or quantifier-free types more generally). 

When $p$ is nonorthogonal to $\mc C$, an initial segment of a Morley sequence of $p$ is not weakly orthogonal to $\mc C$ - that is $p \not\perp \mc C$ implies $p ^{(l)} \not \perp ^w \mc C$ for some $l$ - a main application of \cite{kamensky2024binding} is to bound the value of $l$, which follows by connecting $l$ to the degree of generic transitivity of a certain group action:
\begin{lem} \label{fact1} \cite[Lemma 5.3]{kamensky2024binding}
    Suppose that $SU(p) > 1 $ and $p$ is qf-$\mc C$-internal and weakly qf-$\mc C$-orthogonal. Then let $G=Aut_{qf}(p/\mc C)$. If $q^{(d)}$ is weakly qf-$\mc C$-orthogonal, then $G$ acts transitively and generically $d$-transitively on $p(\mc U).$
\end{lem}

\begin{defn}
Let $p \in S_{qf} (k)$ be stationary. $p$ is \emph{qf}-internal to $\mc C$ if for all $a \models p$ there is a difference field $K$ extending $k$ with $a \ind _k K$ and $a \in K \langle c \rangle$ for some $c \in \mc C$. 
\end{defn}

\begin{lem} \label{rationalreduction} \cite[Proposition 2.11]{kamensky2024binding} Let $p \in S_{qf} (k) $, with $k$ an algebraically closed difference field. Then if $p$ is non-orthogonal to $\mc C$, there is a rational map $p \rightarrow q$ with $q$ the generic rational type of a positive dimensional $\sigma$-variety over $k$ with $q$ qf-$\mc C$-internal. 
\end{lem}

The main thrust of \cite{kamensky2024binding} is to develop the difference Galois theory of qf-$\mc C$-internal types, $p \in S_{qf}(k)$. What they show is that the collection of of automorphisms of $p(\mc U)$ which fix $k$ and $\mc C$ pointwise has the structure of a definble group acting definably on $p$. Because $p$ is $\mc C$-internal, this action is definably isomorphic (over some difference field extension) to the action of a definable group on a definable set in $\mc C$. The general difficulty is that we don't have tight control of this isomorphism, so an analysis of the group theoretic properties of the action becomes necessary. 

We now recall a definition which can be made in various degrees of generality (e.g. supersimple finite rank), but we will restrict to the setting of difference fields, mirroring for instance, section 1.9 of \cite{chatzidakis2008difference}
\begin{defn}
    Let $a$ be a finite rank type over $k$ an algebraically closed difference field. A sequence of tuples $(a_1, \ldots, a_n)$ is called a \emph{semi-minimal} analysis of $a$ over $k$ if $\acl(ka)= \acl(ka_1 a_2\ldots a_n)$ and for each $i$, $\tp(a_i /Ka_1 \ldots a_i)$ is qf-internal to the set of conjugates of a type of $SU$-rank 1.
\end{defn}

\begin{lem} \cite[section 1.9]{chatzidakis2008difference} \label{semimin} Every type $p$ of finite SU-rank has a semi-minimal analysis. The type $p$ is $1$-based if the collection of all $SU$-rank $1$ types associated its semi-minimal analysis are $1$-based.
\end{lem}

Chatzidakis and Hurshovski  proved the following dichotomy theorem for types in $ACFA$ of $SU$-rank $1$. 

\begin{lem} \label{dichot} \cite[5.10]{Chatzidakis99modeltheory} Types $p$ in $S(k)$ which are orthogonal to $\mc C$ are $1$-based. 
\end{lem}

Chatzidakis and Hrusvhovski characterize nontrivial minimal 1-based types - we have slightly restated the result to use our language and specialized to characteristic zero. 

\begin{lem} \cite[5.12]{Chatzidakis99modeltheory} \label{minnontriv} Let $p \in S_{qf} (k)$ be a non-trivial minimal $1$-based type over an algebraically closed $\sigma$-field $k$. Then $p$ is nonorthogonal to the generic type of a minimal subgroup $H$ of $\m G_m$ or a simple abelian variety defined over $k$. 
\end{lem}

From the proof of \cite[5.12]{Chatzidakis99modeltheory}, one can see that the group $H$ and the formula witnessing the the nonorthogonality of the types are quantifier free - in this case, the types become interalgebraic over over a single realization of $p$.

Finally, we will use several facts about the fixed field $\mc C$ of $\mc U \models ACFA_0$:

\begin{lem} \cite[Section 1]{Chatzidakis99modeltheory} \label{stabembedpsf} 
The fixed field $\mc C$ is psuedofinite and stably embedded. 
\end{lem}

By the stable embeddedness of Proposition \ref{stabembedpsf}, when $S=(V,f)^{\sharp}$ is qf-$\mc C$-internal, the action of $G$ on $S$ is definably isomorphic to a definable group acting on a definable set in the structure $\mc C$ in the language of rings. Groups definable in pseudofinite fields are characterized by Hrushovski and Pillay \cite{hrushovski1994groups}: 

\begin{prop} \label{psfgroups}
    Let $G$ be a group definable in a pseudofinite field $F$. Then there is an algebraic group $H$ defined over $F$ and a definable virtual isogeny between $G$ and $H(F)$. 
\end{prop}

A definable virtual isogeny of groups $G, \, H$ is an definable isogeny between definable finite index subgroups $G_1 \leq G$ and $H_1 \leq H$.

\section{Psuedofinite permutation groups} \label{four} 
The following result is a slight rephrasing of the classification of $3$-transitive finite group actions \cite[Theorem 5.2, page 625]{cameron1995permutation}. 

\begin{thm}\label{3fin} Suppose that $G$ is a $3$-transitive finite group. Then $G$ must be one of the following: 
\begin{enumerate}
    \item $M_{11}$, $M_{12}$, $M_{22}$, $M_{23}$, $M_{24}$, or $V_{16}.A_7$
    \item $S_n$ or $A_n$ 
    \item $AGL(d,2)$ acting by affine transformations of a $d$-dimensional vector space over $\m F_2$.
    \item $G$ is a subgroup of $P\Gamma L(2,q)$ containing $PSL(2,q)$ and acting in the natural way on $\m P^1 (\m F_q )$.  
\end{enumerate}
The actions in $(3)$ and $(4)$ are not $4$-transitive. 
\end{thm} 

We won't discuss the groups in (1) except to mention that they are specific finite groups. Some of these groups have a higher degree of transitivity than $3$. The groups $S_n$ and $A_n$ are the symmetric and alternating groups, respectively. $AGL (d,2)$ denotes the group of affine transformations of a $d$-dimensional vector space over $\m F_2$. The group $P \Gamma L(2,q)$ is the semi-direct product of $PGL(2,q)$ and the group of automorphisms of $\m F_q$ over $\m F_p$ where $q=p^n.$

\begin{lem} \label{psuint}
Let $G$ be an infinite pseudofinite definable group which has finite SU-rank and acts $3$-transitively on a definable set $X$. Then the action is not $4$-transitive.
More particularly, $X \cong \m P^1 (\mc F)$, $G=PGL_2(F)$ or $PSL_2(F)$ for some psudofinite field $F$. 
\end{lem}   

\begin{proof}
Any such group $G$ is elementarily equivalent  to an ultraproduct of finite groups $\prod _{\mc F} G_i$ where $\mc F$ is a nonprinciple ultrafilter and $\{i \, | \, |G_i|>n\} \in \mc F$ for any $n \in \m N$, and we use Theorem \ref{3fin}. We can disregard the groups in (1) above. Infinite ultraproducts of the alternating groups, symmetric groups, or $AGL (d,2)$ do not satisfy the hypotheses (e.g. argue via centralizer dimension), we are reduced to case (4). 

So the group is elementarily equivalent to an ultraproduct of subgroups of $P\Gamma L_2 (q)$ for where wlog $q$ is increasing and the action is the natural one on $\m P^1 (\m F_q)$. By a result of \cite{point1999ultraproducts}, this group is isomorphic to $G (\prod _{\mc F} \m F_q)$ acting on $\m P^1 (\prod _{\mc F} \m F_q).$ Then \cite[Theorem 5.6]{zou2020pseudo} implies that $G=PGL_2(F)$ or $PSL_2(F)$.
\end{proof}

\begin{rem}
The appeal to \cite[Theorem 5.6]{zou2020pseudo} wraps up the proof cleanly, but for our main results we ultimately only need to know that the group action can not be $4$-transitive. It might also be useful in some applications to note that if such a group action is $3$-transitive, then it is sharply $3$-transitive. 

We will, in Section \ref{any3sec} only use the previous Lemma in the case that the group action is definable in $\mc C$, the fixed field of our difference field. In that case, one might be able to replace the use of Theorem \ref{3fin}, by a more careful analysis appealing to the classification of definable groups in pseudofinite fields (Proposition \ref{psfgroups}) and some type of O'Nan-Scott argument in the pseudofinite supersimple finite rank context, an area that is under active development \cite{karhumaki2025primitive}. 
\end{rem}

\section{When any four generic solutions are independent} \label{any3sec}
Throughout this section, $p$ is a rational $\sigma$ type over $k$, an algebraically closed difference field. 

\begin{defn} We say that a rational $\sigma$-type $p$ satisfies $C_n$ if for $a_1, \ldots , a_n$ distinct realizations of $p$, $a_i 
\ind_k a_1, \ldots , a_{i-1}$ for each $i=2, \ldots , n$.
\end{defn} 

So, in general $\ldots  C_n \implies C_{n-1} \implies \ldots \implies C_2$. Our main result is to prove most of the converse implications in this chain by showing that $C_4 \implies C_n$ for all $n$. As we mentioned in the introduction, because autonomous equations have additional obvious sources of invariant subvarieties given by the iterates of graphs of the dynamical system, it is easy to see that $C_2$ fails. Even in the non-autonomous case, some aspects of our proof might be possible to improve on, so that in the case that the difference equation is of order greater than one, we don't know examples precluding $C_3$ rather than $C_4$. 

\begin{lem} \label{any2}
Every rational $\sigma$-type $p$ satisfying $C_2$ is either 
\begin{itemize} 
\item almost qf-$\mc C$-internal and weakly qf-orthogonal to $\mc C$ 
\item minimal and locally modular. 
\end{itemize} 
\end{lem}

\begin{proof}
Consider the semi-minimal analysis $(a_1, \ldots, a_n)$ of $a \models p$ and suppose that the dimension of $p$ is $d$. We claim that $n=1$. If $n>1$, then consider a realization of $b_n \models \tp (a_n/ ka_1 \ldots a_{n-1} )$ not equal to $a_n$. But then considering  $(a_1, a_2 \ldots, a_n, a_1, \ldots a_{n-1}, b_n)$ This tuple is interalgebraic with a tuple of realizations of $p$, $(a,b)$ and by the hypotheses it follows that $td_k (k (a,b)) = td_k (k \langle a, b \rangle ) = td _k (k \langle a_1, a_2 \ldots, a_n, a_1, \ldots a_{n-1}, b_n \rangle )$, but $d < td _k (k \langle a_1, a_2 \ldots, a_n, a_1, \ldots a_{n-1}, b_n \rangle ) <2d$, contradicting $C_2.$ 

So, the type $p$ can be assumed to be internal to the collection of conjugates of a fixed type with $SU$-rank $1$. 
First, assume that this type is $1$-based, then so is $p$ by Fact \ref{semimin}. Since $p$ is internal to a collection of minimal $1$-based types, any forking extension of $p$ is also orthogonal to $\mc C$ and thus $1$-based by Lemma \ref{dichot}. Now we claim that if $C_2$ holds then $p$ is minimal.

To see this, note that if not, then for some $n$ there is a sequence of \emph{distinct} realizations of $p$, $a_1 \ldots a_n$ so that $tp (a_n / k a_1 \ldots a_{n-1})$ is a nonalgebraic forking extension of $p$. But then by $1$-basedness, canonical base is algebraic over over a single realization of this type which extends $p$, so we can see $C_2$ fails already. 

If the type $p$ is not $1$-based, then by Fact \ref{dichot}, $p$ is non-orthogonal to $\mc C$. Then by Fact \ref{rationalreduction} there is a $k$-rational map $f:p\rightarrow q$ to a rational $\sigma$-type which is qf-$\mc C$-internal. If the fibers of the map $f$ are infinite, then $p$ has a forking extension over the image $f(a)$ of a realization of $p$. Weak orthogonality must hold since otherwise $C_2$ fails witnessed by the invariant rational function giving a map to $(\m A^1, id)$. 
\end{proof}

\begin{lem} \label{any3modular}
If $p \in S(k)$ is minimal and orthogonal to $\mc C$, then $C_4 \implies C_n$ for all $n$. 
\end{lem}

\begin{proof}
If $C_n$ fails for some $n$, then perhaps over some field extension, $K$, there is a forking extension of $p^{(2)}$ in which for a realization $(a,b)$, we have $a \neq b$, but by Fact \ref{dichot}, $p^{(2)}$ is $1$-based, so there is a canonical base for this forking extension over a single realization of $p^{(2)}$. But then the realization of the type $p^4$ which realizes this forking extension over the canonical base shows that $C_4$ fails. 
\end{proof}

\begin{lem} \label{any3}
Let $p$ be the $\sigma$-type which is the qf-generic of $S = (V,f)^ \sharp$. Suppose that $C_n$ fails for some $n$. Then $C_4$ fails of $p$.  
\end{lem}

\begin{proof}
By Lemmas \ref{any2} and \ref{any3modular}, we can assume that $p$ is qf-$\mc C$-internal. Suppose, for a contradiction, that $C_4$ holds. By Fact \ref{fact1}, we must have that $G = Aut_{qf}(p(\mc U)/ \mc C)$ acts generically $4$-transitively. In this case, by \cite[Theorem 5.4]{kamensky2024binding} the type $p$ is isolated by a quantifier free formula (given by excluding a proper subvariety of $V$). The condition $C_4$ implies that any four realizations are independent, and thus have the same qf-type over $k$. Thus, $G$ acts $4$-transitively. By Proposition 2.10 of \cite{kamensky2024binding}, $p$ is qf-interdefinable (birational) with a tuple of elements of the fixed field over some $\sigma$-field extension of $k$. Thus, $G$ acting definably on $p$ is definably isomorphic to the action of a definable group acting definably on a definable subset of the fixed field. But then by Lemma \ref{stabembedpsf}, we are in the setting of Lemma \ref{psuint}, so we can see that this is impossible. 
\end{proof}

\begin{thm} \label{any3>1}
Let $p$ be the generic type of $(V, \phi)^\sharp$. If $C_4$ holds of $p$ then $p$ is minimal and $C_n$ holds of $p$ for all $n \in \m N.$
\end{thm}
\begin{proof}
By Lemma \ref{any3}, it must be that $p$ is orthogonal to $\mc C$ and $C_4$ holds. But then Lemmas \ref{any2} and \ref{any3modular} imply that $C_n$ holds for all $n \in \m N$. 
\end{proof}

\subsection{Translation back to invariant subvarieties} We now mention the translation of the main theorems of this section back to the language of invariant subvarieties used in the introduction. Supposing that $C_n$ fails for $p,$ the generic type of $(V, \phi)$ means that there is a difference-algebraic relation over $k$ between an $n$-tuple of distinct realizations of the type $p$, $(a_1, \ldots , a_n)$. Since $\sigma (a_i) \in k(a_i)$, this gives an algebraic relation over $k$ between $a_1, \ldots , a_n$. Then the Zariski-closed set $W$ which is the zero set of the polynomials over $k$ which vanish at $(a_1, \ldots , a_n)$ is an irreducible invariant $k$-subvariety of $V^n$. Since $a_i \models p$, the projection to the $i^{th}$ copy of $V$ is dominant. Since the realizations $a_1, \ldots , a_n$ are distinct, the variety $W$ is not contained in any generalized diagonal. So, $W \subset V^n$ is a non-diagonal irreducible invariant subvariety. Thus we can see Theorem \ref{any3>1} implies Theorem \ref{mainthm}.

\subsection{The degree of nonminimality and forking subvarieties}
We now prove Theorems \ref{nmdeg} and \ref{nmdegnon}. 
   \begin{thm} \label{2qfisfun}
    Suppose $(X, \phi)$ is an autonomous algebraic dynamical system. If $(X^n, \phi^{\times n})$ has a nonminimal forking subvariety for some $n$, then already $(X^3, \phi ^{\times 3})$ has a nonminiml forking subvariety. 
   \end{thm} 

\begin{proof}
First suppose that $p$ is not orthogonal to $\mc C$. By Lemma \ref{rationalreduction}, let $g: p \rightarrow q$ be a $k$-rational map such that $q$ is positive dimensional and $qf$-$\mc C$-internal. We can assume that $p=q$, since if $\dim p= \dim q$ in the proof, then it is fine to replace $p$ by $q$. If the difference in their dimensions is positive let $a_1$ be a generic realization of the quantifier free type of $a$ over $k \langle g(a) \rangle $. Then the locus of $(a, a_1)$ yields an invariant forking subvariety of $(X^2 , \phi^{\times 2})$. 

When $p$ is itself $qf$-$\mc C$-internal, there is, by \cite[Corollary B]{bell2024invariant}, a $\phi$-invariant rational function $f : X^2 \rightarrow \m A^1$. The same argument as above, choosing two generic points of a generic fiber of $f$ yields an invariant forking subvariety. The fiber of $f$ above a point $c \in \mc C$ containing an irreducible invariant component of a fixed dimension is a definable property which holds of at the generic point of the image of $f$ in $\mc C$. Thus it holds on a positive measure subset of $\mc C$. If this set of positive measure contained a point in the algebraic closure of the empty set, the fiber over this point would give a forking subvariety of $X^2$. As it stands one can only see a forking subvariety of $X^3$ obtained by setting the induced invariant rational function on the product first and second copies of $X$ equal to the induced invariant rational function on the product of the first and third copies of X.\footnote{There was a mistake at this point in early versions of this manuscript, in which it a fiber above a point in $\mc C$ algebraic over the empty set was considered.} 

So, we are reduced to the case that $p$ is orthogonal to $\mc C$. We can assume that $p$ is semi-minimal. Then if there is for some $n$, an invariant forking subvariety of $V^n$, then $p$ has a nonalgebraic forking extension, $q$. But in this case, by Lemma \ref{dichot}, $q$ is $1$-based - so the canonical base of $b\models q$ is algebraic over a single realization of the type, $a$. So the locus of $(a,b)$ over $k$ gives an invariant forking subvariety. 
\end{proof}

When we drop the assumption that the $\sigma$-variety is autonomous, each part of the proof of the previous theorem works identically except for the use of \cite[Corollary B]{bell2024invariant}. Replacing \cite[Corollary B]{bell2024invariant} with the analogous result of \cite{kamensky2024binding} in the nonautonomous case with the worse bound of $\dim V +3$ yields a proof of Theorem \ref{nmdegnon}.

\subsection{Further work}
There are many interesting variants of the results pursued above which we leave for future work. Perhaps the most interesting question is the determination of the \emph{forking degree} or the \emph{degree of nonminimality} \cite{baldwin2024simple, oberwolfachFdeg} for quantifier free types in $ACFA$. It seems likely that a careful analysis of binding groups might yield an absolute bound in the nonautonomous case.

Can one show that Theorem \ref{any3>1} is optimal in the sense that $4$ can not be replaced with $3$? In the case that $p$ has dimension one, this should be precluded by, for instance, the difference Riccati equations of \cite{nishioka2018differential}. When the dimension of the $\sigma$-variety $(X, \phi)$ is larger than one, for the binding group action to be $3$-transitive, it must be that the field appearing in the classification result of Lemma \ref{psuint} must be the fixed field of $\sigma^{\dim X  }$. We don't have any particular obstruction to this happening, but we also don't know an example. Thanks to Remi Jaoui for pointing out this issue. 

\bibliographystyle{plain}
\bibliography{research}
\end{document}